\documentclass{article}
\usepackage{amsmath, amsthm, amssymb, amsfonts, ascmac, color}
\usepackage[colorlinks = True]{hyperref}
\usepackage[margin = 2cm]{geometry}
\theoremstyle{definition}
\newtheorem{definition}{Definition}[section]
\newtheorem{lemma}[definition]{Lemma}
\newtheorem{theorem}[definition]{Theorem}
\newtheorem{corollary}[definition]{Corollary}
\newtheorem{remark}[definition]{Remark}

\newtheorem{proposition}[definition]{Proposition}

\begin{document}
\title{Convex inequalities and their applications to relative operator entropies}
\author{Shigeru Furuichi$^1$\footnote{E-mail:furuichi.shigeru@nihon-u.ac.jp}, Hamid Reza Moradi$^2$\footnote{hrmoradi@mshdiau.ac.ir} and Supriyo Dutta$^3$\footnote{dosupriyo@gmail.com}\\
$^1${\small Department of Information Science, College of Humanities and Sciences,}\\
{\small Nihon University, Setagaya-ku, Tokyo, 156-8550, Japan}\\
$^2${\small Department of Mathematics, Payame Noor University (PNU),}\\ 
{\small P.O.Box, 19395-4697, Tehran, Iran.}\\
$^3${\small Department of Humanity and Sciences (Mathematics), SR University,}\\
{\small Warangal Urban, Telengana, India - 506371.}}
\date{}
\maketitle
{\bf Abstract.} A considerable amount of literature in the theory of inequality is devoted to the study of Jensen's and Young's inequality. This article presents a number of new inequalities involving the log-convex functions and the geometrically convex functions. As their consequences, we derive the refinements for Young's inequality and Jensen's inequality. In addition, the operator Jensen's type inequality is also developed for conditioned two functions. Utilizing these new inequalities, we investigate the operator inequalities related to the relative operator entropy.
\vspace{3mm}

\vspace{3mm}

%Primary 26D15, secondary 26B25, 26E60.  
\vspace{3mm}

	\section{Introduction}
	
		Throughout this article $I\subset \mathbb{R}$ and $J\subset (0, \infty)$ denote two intervals. A function $f:I \rightarrow \mathbb{R}$ is said to be a convex function, if it satisfies the well-known Jensen's inequality, which is 
		\begin{equation}\label{15}
		f\left( \left( 1-t \right)a+tb \right)\le \left( 1-t \right)f\left( a \right)+tf\left( b \right),
		\end{equation}
		for all $a,b\in I $ and all $0\le t \le 1$. In general, we can express it as
		\begin{equation}\label{JI}
		f\left(\sum\limits_{i=1}^n t_ia_i\right) \le \sum\limits_{i=1}^n t_i f(a_i),
		\end{equation}
		where $a_1,a_2,\cdots,a_n \in I$ and $t_1,t_2,\cdots,t_n \ge 0$ with $\sum\limits_{i=1}^nt_i=1$. The Young's inequality \cite{8} is a consequence of Jensen's inequality, which states that
		\begin{equation}\label{14}
		{{a}^{1-t}}{{b}^{t}}\le \left( 1-t \right)a+tb,
		\end{equation}
		where $a, b >0$ and $0 \leq t \leq 1$. It is generated from inequality (\ref{15}) applying $f\left( x \right)=-\log x$. A function $f:I \rightarrow (0,\infty)$ is described as a log-convex function, if $f\left((1-t)a+tb\right)\le f^{1-t}(a)f^t(b)$ for all $a,b\in I$ and all $0\le t \le 1$. In addition, we mention a function $f:J\to(0,\infty)$ as a geometrically convex function if $f\left(a^{1-t}b^t\right)\le f^{1-t}(a)f^t(b)$ for all $a,b\in I$ and all $0\le t \le 1$.
		
		A self-adjoint operator $A$ is said to be positive, if $\langle Ax, x \rangle \geq 0$ for all $x \in \mathcal{H} $, where $\mathcal{H}$ is a Hilbert space with inner product $\langle ., . \rangle$. We denote the identity operator in $\mathcal{H}$ by ${{\mathbf 1}_{\mathcal H}}$. Given two self-adjoint operators $A$ and $B$ we denote $A \geq B$ ($A > B$) when $A - B$ is positive (strictly positive) operator. Therefore $A \geq 0$ ($A > 0$) indicates $A$ is a positive (strictly positive) operator.  
		
		This article is distributed as follows. Section 2 illustrates our results on log-convex functions and the geometrically convex functions. Theorem \ref{2} and its corollaries present a refinement of Jensen inequality. Theorem \ref{concave_and_convex_lemma} discusses operator monotonicity and concavity. Section 3 presents the applications of the results developed in section 2. Here, we consider the deformed logarithmic function and operator relative entropy. In this section, we construct several operator inequalities applicable in entropy theory. Remark 3.9 presents a refinement of Young's inequality.

	\section{Main results}\label{sec2}
	
		We denote the weighted arithmetic mean and geometric mean of two positive real numbers $a$ and $b$ by $a\nabla_v b:=(1-v)a+vb$ and $a\sharp_vb:=a^{1-v}b^v$, respectively. The first result of this article describes a bound on the ratio of these two quantities, which is as follows:
	
			\begin{proposition}\label{theorem_2.1}
			Given two positive real numbers $a$ and $b$ we have
			\begin{equation}\label{theorem2.1_ineq}
			\exp\left(n\left(1-\left(\frac{a\nabla_v b}{a\sharp_v b}\right)^{-1/n}\right)\right) \le \frac{a\nabla_v b}{a\sharp_v b} \le \exp\left(n\left(\left(\frac{a\nabla_v b}{a\sharp_v b}\right)^{1/n}-1\right)\right),
			\end{equation}
			where $0\le v\le 1$, and $n \in \mathbb{N}$.
		\end{proposition}
		\begin{proof}
			Recall that, for any differentiable convex function $f$ on the interval $J$, we can write
			\begin{equation}\label{9}
				f'\left( s \right)\left( t-s \right)\le f\left( t \right)-f\left( s \right)\le f'\left( t \right)\left( t-s \right),
			\end{equation}
			for any $t \ge s > 0$. Applying $f\left( x \right)=-\log x$ in the above inequality we can derive that
			\[1-\frac{s}{t}\le \log \frac{t}{s}\le \frac{t}{s}-1\]
			which implies
$$
			1-\frac{1}{t}\le \log t \le t-1, \,\,\,\,(t>1),
$$
			for $s= 1$. In addition, applying $t = 1$ we can deduce that 
$$
			1-\frac{1}{s}\le \log s \le s-1, \,\,\,\,(0<s\le 1).
$$
		Combining we get the well known inequality:
			\begin{equation}\label{8}
			1-\frac{1}{t}\le \log t \le t-1, \,\,\,\,(t>0).
			\end{equation}
			Let $x > 0$. Eliminating $t:=x^{1/n}>0$ in \eqref{8}, we obtain
			\begin{equation}\label{ineq01_proof_theorem2.1}
			n\left( 1-x^{-1/n}\right)\le \log x\le n\left(x^{1/n}-1\right),\,\,\,\,(x>0).
			\end{equation}
			Thus, we have the desired result by replacing $x:=\dfrac{a\nabla_vb}{a\sharp_vb}$ in the above inequality.
		\end{proof}
	
	\begin{remark} 
		Consider two sequences of functions $\{a_n(x)\}$ and $\{b_n(x)\}$ defined by $a_n(x):= n\left(x^{1/n}-1\right)$ and $b_n(x):= n\left(1-x^{-1/n}\right)$, respectively. We observe that $\{a_n\}$ is a monotone decreasing sequence for $x \geq 1$, since the function $f(y):= ny^{n+1} - (n+1)y^n + 1$ where $y = x^{\frac{1}{n(n + 1)}}$ is monotone increasing for $y \geq 1$. Similarly, defining $g(z):= nz^{-n-1} - (n+1)z^{-n} + 1$ where $z:=x^{\frac{1}{n(n+1)}}>0$ we can prove $\{b_n\}$ is a monotone increasing sequence when $x \geq 1$. Moreover, $\lim\limits_{n\to\infty} a_n(x) = \lim\limits_{n\to\infty} b_n(x) =\log x$. Thus, both sides in \eqref{theorem2.1_ineq} converges to $\dfrac{a\nabla_vb}{a\sharp_vb}$ in the limit $n \to \infty$.
	\end{remark}

	\begin{theorem}\label{thm}
		Given a function $f:I\to\mathbb{R}$ and two real numbers $s$ and $t \in I$ with $t > s$ the following inequality holds:
		$$\min\left\{ \frac{\left(f(t)-f(s)\right)^2}{t-s},t-s\right\}\leq |f(t)-f(s)|\leq \max\left\{ \frac{\left(f(t)-f(s)\right)^2}{t-s},t-s\right\}.$$
	\end{theorem}

	\begin{proof}
		Given any real number $\alpha \leq 1$ we have $\alpha^2\le \alpha \le 1$, that is $\min\{\alpha^2,1\}=\alpha^2 \leq \alpha \leq 1 = \max\{\alpha^2,1\}$. Similarly, considering $\alpha\geq 1$, we observe $\alpha^2 \geq \alpha \ge 1$, which leads us to $\min\{\alpha^2,1\}=1\le \alpha \leq \alpha^2 =\max\{\alpha^2,1\}$. Combining we get 
		$$\min\{\alpha^2,1\}\leq \alpha\leq \max\{\alpha^2,1\}$$ 
		for all $\alpha\in\mathbb{R}$. Replacing $\alpha:=\dfrac{|f(t)-f(s)|}{t-s}$ in the above inequality we find
		$$\min\left\{ \left(\frac{f(t)-f(s)}{t-s}\right)^2,1\right\}\leq \frac{|f(t)-f(s)|}{t-s}\leq \max\left\{ \left(\frac{f(t)-f(s)}{t-s}\right)^2,1\right\}.$$
		We complete the proof by multiplying $t - s > 0$ with all the terms in the above inequality.
	\end{proof}

	\begin{remark}\label{3}
	\begin{itemize}
		\item[(i)] Consider $f$ is a monotonically increasing function, that is $f(t) > f(s)$ for $t > s$. Hence, Theorem \ref{thm} takes the form
		$$\min\left\{ \frac{\left(f(t)-f(s)\right)^2}{t-s},t-s\right\}\leq f(t)-f(s)\leq \max\left\{ \frac{\left(f(t)-f(s)\right)^2}{t-s},t-s\right\}.$$
		\item[(ii)] Moreover, Theorem \ref{thm} leads us to
		\[\frac{\min \left\{ {{\left( f\left( t \right)-f\left( 0 \right) \right)}^{2}},{{t}^{2}} \right\}}{t}\le f\left( t \right)-f\left( 0 \right)\le \frac{\max \left\{ {{\left( f\left( t \right)-f\left( 0 \right) \right)}^{2}},{{t}^{2}} \right\}}{t},\]
		when $s = 0 \in I$ and $f : I \to \mathbb{R}$ is a monotone increasing function.
		
		\item[(iii)]For $\alpha >0$, $p\ge 1$ and $q \le 1$, we have $\min\{\alpha^p,\alpha^q\}\le \alpha \le \max\{\alpha^p,\alpha^q\}$. Now, given a monotone increasing function $f:I\to \mathbb{R}$ as well as two points $s, t\in I$ with $t > s$, we have
		$$
		\min\left\{\frac{\left(f(t)-f(s)\right)^p}{(t-s)^{p-1}},\frac{\left(f(t)-f(s)\right)^q}{(t-s)^{q-1}}\right\}\le f(t)-f(s)\le \max\left\{\frac{\left(f(t)-f(s)\right)^p}{(t-s)^{p-1}},\frac{\left(f(t)-f(s)\right)^q}{(t-s)^{q-1}}\right\}.
		$$
		Proof is similar to Theorem \ref{thm}. 
		%The case $p=2, q=0$ recovers the second statement in Theorem \ref{thm}. 
		The positiveity of $\alpha$ can be dropped for the case $p=2, q=0$.
	\end{itemize}
	\end{remark}
	
	\begin{corollary}\label{corollary_2.1}
		Consider a convex function $f : J \rightarrow \mathbb{R}$ and $0\le t\le 1$. For ${{x}_{1}},{{x}_{2}},\ldots ,{{x}_{n}}\in J$ and positive real numbers ${{w}_{1}},{{w}_{2}},\ldots ,{{w}_{n}}$ with
		$\sum\limits_{i=1}^{n}{{{w}_{i}}}=1$ we have
		\[f\left( \sum\limits_{i=1}^{n}{{{w}_{i}}{{x}_{i}}} \right)+\frac{{{\psi }}(t)}{t}\le \sum\limits_{i=1}^{n}{{{w}_{i}}f\left( {{x}_{i}} \right)},\]
		where
		\[{{\psi }}(t)=\min \left\{ {{\left( \sum\limits_{i=1}^{n}{{{w}_{i}}f\left( t{{x}_{i}}+\left( 1-t \right)\sum\limits_{j=1}^{n}{{{w}_{j}}{{x}_{j}}} \right)}-f\left( \sum\limits_{i=1}^{n}{{{w}_{i}}{{x}_{i}}} \right) \right)}^{2}},{{t}^{2}} \right\}.\]
	\end{corollary}
	\begin{proof}
		We can prove that the function 
		\[g\left( t \right)=\sum\limits_{i=1}^{n}{{{w}_{i}}f\left( t{{x}_{i}}+\left( 1-t \right)\sum\limits_{j=1}^{n}{{{w}_{j}}{{x}_{j}}} \right)}\]
		is a monotonically increasing function on the interval $\left[ 0,1 \right]$ \cite[Theorem 2.1]{0}. Moreover,
		\[g\left( 0 \right)=f\left( \sum\limits_{i=1}^{n}{{{w}_{i}}{{x}_{i}}} \right)\text{ and }g\left( 1 \right)=\sum\limits_{i=1}^{n}{{{w}_{i}}f\left( t{{x}_{i}} \right)}.\]
		Therefore, we get the desired result by applying Remark \ref{3} (ii).
	\end{proof}
	
	\begin{remark}
	Corollary \ref{corollary_2.1} leads us to an improvement of the Jensen's inequality. It also assists us to revise the arithmetic-geometric mean inequality. Applying $f(x):=-\log x$ in Corollary \ref{corollary_2.1} we obtain
$$
		0 \le \min\left\{\left(\log \dfrac{A(w_i,x_i)}{G(w_i,x_i\nabla_{1-t}A(w_i,x_i))}\right)^2,t^2\right\}\times\frac{1}{t} \le A(w_i,x_i) -G(w_i,x_i) ,
$$
	where $A$ and $G$ are the weighted arithmetic mean and geometric mean, respectively defined by
	$$
	A(w_i,x_i):=\sum_{i=1}^nw_ix_i,\quad G(w_i,x_i):=\prod_{i=1}^nx_i^{w_i}.
	$$
	\end{remark}

	\begin{theorem}\label{1}
		If $f:J\to(0,\infty)$ be a differentiable log-convex function, then for any $s,t\in J$,
		\begin{equation}\label{theorem3_1st_statement}
		\exp \left( \frac{f'\left( s \right)}{f\left( s \right)}\left( t-s \right) \right)\le \frac{f\left( t \right)}{f\left( s \right)}\le \exp \left( \frac{f'\left( t \right)}{f\left( t \right)}\left( t-s \right) \right).
		\end{equation}
			If $f:J\to(0,\infty)$ is a differentiable log-concave function, then for any $s,t\in J$,
		\begin{equation}\label{theorem3_2nd_statement}
		\exp \left( \frac{f'\left( t \right)}{f\left( t \right)}\left( s-t \right) \right)\le \frac{f\left( t \right)}{f\left( s \right)}\le \exp \left( \frac{f'\left( s \right)}{f\left( s \right)}\left( s-t \right) \right).	
		\end{equation}
	\end{theorem}
	\begin{proof}
		A function $f: J \rightarrow \mathbb{R}$ is log-convex if $\log f$ is convex. Inequality \eqref{9} indicates
		\[\left( \log f\left( s \right) \right)'\left( t-s \right)\le \log f\left( t \right)-\log f\left( s \right)\le \left( \log f\left( t \right) \right)'\left( t-s \right),\]
		that is
		\[\frac{f'\left( s \right)}{f\left( s \right)}\left( t-s \right)\le \log \frac{f\left( t \right)}{f\left( s \right)}\le \frac{f'\left( t \right)}{f\left( t \right)}\left( t-s \right).\]
		Applying the exponential function we get
		\[\exp \left( \frac{f'\left( s \right)}{f\left( s \right)}\left( t-s \right) \right)f\left( s \right)\le f\left( t \right)\le \exp \left( \frac{f'\left( t \right)}{f\left( t \right)}\left( t-s \right) \right)f\left( s \right).\]
		
		The function $f$ is called log-concave if $-\log f$ is convex \cite[Definition 1.3.1]{NP2018}. Applying the inequality \eqref{9} with $-\log f$ and exchanging $s$ and $t$ we obtain the second statement.
	\end{proof}

	\begin{theorem}\label{2}
		If $f:J\to \left( 0,\infty  \right)$ be a differentiable geometrically convex function, then for any $s,t\in J$,
		\[{{\left( \frac{t}{s} \right)}^{\frac{sf'\left( s \right)}{f\left( s \right)}}}\le \frac{f\left( t \right)}{f\left( s \right)}\le {{\left( \frac{t}{s} \right)}^{\frac{tf'\left( t \right)}{f\left( t \right)}}}.\]
	\end{theorem}
	\begin{proof}
		The function $f: J \rightarrow \mathbb{R}$ is geometrically convex if and only if $\log f\left( \exp\left( x \right) \right)$ is convex on $\mathbb{R}$. Hence, from \eqref{9}, we infer that
		\[\left( \log f\left( \exp s \right) \right)'\left( t-s \right)\le \log \frac{f\left( \exp t \right)}{f\left( \exp s \right)}\le \left( \log f\left( \exp t \right) \right)'\left( t-s \right).\]
		We see that
		\[\frac{(\exp s)f'\left( \exp s \right)}{f\left( \exp s \right)}\left( t-s \right)\le \log \frac{f\left( \exp t \right)}{f\left( \exp s \right)}\le \frac{(\exp t)f'\left( \exp t \right)}{f\left( \exp t \right)}\left( t-s \right).\]
		Replacing $s$ and $t$ by $\log s$ and $\log t$ respectively, we get the required result.
	\end{proof}

	\begin{corollary}\label{corollary_9}
		Let $g:J\to(0,\infty)$ be a differentiable geometrically convex function and $a,b\in J$. Then for any $0\le t\le 1$,
		\begin{equation}\label{theorem_9_eq01}
		\exp \left( G_0(g;a,b)t \right)\le \frac{g\left( {{a}^{1-t}}{{b}^{t}} \right)}{{{g}^{1-t}}\left( a \right){{g}^{t}}\left( b \right)} \le \exp \left( G_t(g;a,b)t \right),
		\end{equation}
		where 
		$$
		G_t(g;a,b):=\log\frac{g(a)}{g(b)}-a^{1-t}b^t\left(\log \frac{a}{b}\right)\frac{g'(a^{1-t}b^t)}{g(a^{1-t}b^t)}.
		$$
	\end{corollary}
	\begin{proof}
		If $g$ is a geometrically convex function, then $f\left( t \right)=\dfrac{g\left( {{a}^{1-t}}{{b}^{t}} \right)}{{{g}^{1-t}}\left( a \right){{g}^{t}}\left( b \right)}$ is a log-convex function. Indeed we calculate for $0\le s,t\le 1$,
$$
		f\left(\dfrac{s+t}{2}\right)=\frac{g\left(a^{1-\frac{s+t}{2}}b^{\frac{s+t}{2}}\right)}{g^{1-\frac{s+t}{2}}(a)g^{\frac{s+t}{2}}(b)}= \frac{g\left(\sqrt{a^{1-t}b^t}\sqrt{a^{1-s}b^s}\right)}{\sqrt{g^{1-t}(a)g^t(b)}\sqrt{g^{1-s}(a)g^s(b)}}\le \frac{\sqrt{g\left(a^{1-t}b^t\right)}\sqrt{g\left(a^{1-s}b^s\right)}}{\sqrt{g^{1-t}(a)g^t(b)}\sqrt{g^{1-s}(a)g^s(b)}}=\sqrt{f(t)f(s)}.
$$
		Also, $\dfrac{f'(t)}{f(t)}=G_t(g;a,b)$. Taking $s=0$ as $f(0)=1$, we obtain the result Theorem \ref{1}.
	\end{proof}
	
	\begin{remark}
If the function $g:I\to (0,\infty)$ is differentiable log-convex  and monotone increasing, then also we can prove the inequality \eqref{theorem_9_eq01}.
%	From both sides of the above inequalities, we have
%\begin{equation}\label{remark_5_eq01}
%	\left(\log\frac{a}{b}\right)\left(\frac{ag'(a)}{g(a)}-\frac{bg'(b)}{g(b)}\right)\ge 0,\,\,\,(a,b>0)
%\end{equation}
%	 Taking log-convex functions $g(x):=c x$ or $g(x):=\dfrac{1}{x^p},\,\,\,(p>0,\,\,x>0)$ as examples, the equality holds in \eqref{remark_5_eq01}. If we take another log-convex function $g(x):=\exp\left(x^p\right),\,\,\,(p\ge 1)$, then \eqref{remark_5_eq01} is reduced to the natural inequality $\left(a^p-b^p\right)\log\dfrac{a^p}{b^p}\ge 0$ for $a,b>0$ and $p \ge 1$.
	\end{remark}
%	\medskip
%	
%	\bf Do the following inequalities are true for any differentiable geometrically convex function?
%		\[\frac{g'\left( a \right)}{g\left( a \right)}-\log \frac{g\left( b \right)}{g\left( a \right)}\ge 0,\]
%	and
%		\[\frac{g'\left( {{a}^{1-t}}{{b}^{t}} \right)}{g\left( {{a}^{1-t}}{{b}^{t}} \right)}-\log \frac{g\left( b \right)}{g\left( a \right)}\le 0.\]
%	
%	
%	\medskip
	
	\begin{corollary}\label{cor_211}
		Let ${{a}_{1}},{{a}_{2}},\ldots ,{{a}_{n}}\in J$, ${{t}_{1}},{{t}_{2}},\ldots ,{{t}_{n}}$ be non-negative real numbers, such that, $\sum\limits_{i=1}^{n}{{{t}_{i}}}=1$. Denote $A(t_i,a_i):=\sum\limits_{i=1}^nt_ia_i$. Now, we have the following inequalities:
		\begin{itemize}
		\item[(i)]
		If $f:J\to(0,\infty)$ is a differentiable log-convex function, then
\begin{equation}\label{corollary6_eq01}
		L_nf\left(\sum\limits_{i=1}^n t_ia_i\right) \le \sum\limits_{i=1}^n t_i f(a_i) \le R_nf\left(\sum\limits_{i=1}^n t_ia_i\right),
\end{equation}
		where
		$$
		L_n:=\sum_{i=1}^n t_i \exp\left(\frac{f'\left(A(t_i,a_i)\right)\left(a_i-A(t_i,a_i)\right)}{f\left(A(t_i,a_i)\right) }\right)
		\,\,\, {\rm and}\,\,\,\,\,
		R_n:=\sum\limits_{i=1}^nt_i\exp\left(\frac{f'(a_i)\left(a_i-A(t_i,a_i)\right)}{f(a_i)}\right).
		$$
		\item[(ii)] If $f:J\to(0,\infty)$ is a differentiable geometrically convex function, then we also have
\begin{equation}\label{corollary6_eq02}
		\hat{L_n}f\left(\sum\limits_{i=1}^n t_ia_i\right) \le \sum\limits_{i=1}^n t_i f(a_i) \le \hat{R_n}f\left(\sum\limits_{i=1}^n t_ia_i\right),
\end{equation}
		where
		$$
		\hat{L_n}:=\sum\limits_{i=1}^nt_i\left(\frac{a_i}{A(t_i,a_i)}\right)^{\frac{A(t_i,a_i)f'(A(t_i,a_i))}{f(A(t_i,a_i))}}
		\,\,\, {\rm and}\,\,\,\,\,
		 \hat{R_n}:=\sum\limits_{i=1}^nt_i\left(\frac{a_i}{A(t_i,a_i)}\right)^{\frac{a_if'(a_i)}{f(a_i)}}.
		$$
		\end{itemize}
%\end{eqnarray*}
	\end{corollary}
	\begin{proof}
		Replacing $t = {{a}_{i}}$ $\left( i=1,2,\ldots ,n \right)$ in Theorem \ref{1}, we write
		\begin{equation}\label{6}
		\exp \left( \frac{f'\left( s\right)}{f\left(s \right)}\left({{a}_{i}}-s \right) \right)\le \frac{f\left( a_i \right)}{f\left( s \right)}\le \exp \left( \frac{f'\left( a_i \right)}{f\left( a_i \right)}\left({{a}_{i}} -s\right) \right)
		\end{equation}
		for any $i=1,2,\ldots ,n$. Multiplying ${{t}_{i}}$ to both sides in \eqref{6}  and adding over $i$ from $1$ to $n$ we deduce
		\begin{equation}\label{13}
		\sum\limits_{i=1}^{n}{{{t}_{i}}\exp \left( \frac{f'\left( s \right)}{f\left( s\right)}\left( {{a}_{i}}-s \right) \right)}\le \frac{\sum\limits_{i=1}^{n}{{{t}_{i}}f\left( {{a}_{i}} \right)} }{f\left( s \right)}\le \sum\limits_{i=1}^{n}{{{t}_{i}}\exp \left( \frac{f'\left( a_i \right)}{f\left( a_i \right)}\left( {{a}_{i}} -s\right) \right)}.
		\end{equation}
		Putting $s=A(t_i,a_i)$ in \eqref{13}, then we infer the result (i). Similarly, we can prove result (ii) by substituting  $s=A(t_i,a_i)$ and $t=a_i$  in Theorem \ref{1}.
	\end{proof}
	
	\begin{remark}
	A  log-convex function $f:I\to(0,\infty)$ is always a convex function which satisfies the Jensen's inequality mentioned in \eqref{JI}. Now, we have the following observations:
	\begin{itemize}
	\item[(i)] Since $e^{\alpha} >0$ for all $\alpha \in \mathbb{R}$, we observe $L_n>0$ and $R_n>0$ in corollary \ref{cor_211}. Note that, $e^{\alpha}\ge 1$ if and only if  $\alpha \ge 0$. Thus, the first inequality in \eqref{corollary6_eq01} is a refined Jensen's inequality when $L_n\ge 1$. The condition $f'(A(t_i,a_i))(a_i-A(t_i,a_i))\ge 0$ for all $i=1,2,\cdots,n$ is a strong sufficient condition for $L_n \ge 1$. The sign of $a_i-A(t_i,a_i)$ for each $i$ is not determinate, since $A(t_i,a_i)$ is an arithmetic mean. There are examples of the log-convex function $f$ such that $f'(x)\ge 0$ or $f'(x)\le 0$. The function $f(x) = \exp({x^p})$ for $x>0$ and $p\ge 1$ is log-convex and $f'(x) > 0$. Also, the function $f(x) = \frac{1}{x^p}$ for $x>0$ and $p>0$ is another log-convex with $f'(x) < 0$. 
	%Therefore, there exists a possibility such that the first inequality in  \eqref{corollary6_eq01} gives the refinement of Jensen's inequality. 
	The second inequality in  \eqref{corollary6_eq01} is a reverse of Jensen's inequality. 
	
	\item[(ii)] 
	A differentiable monotonically decreasing geometrically convex function $f:J\to(0,\infty)$ is a convex function, since 
	$$
	f\left((1-t)x+ty\right)\le f(x^{1-t}y^t) \le f^{1-t}(x)f^t(y)\le (1-t)f(x)+tf(y).
	$$
	Therefore the first inequality in \eqref{corollary6_eq02} provides a refinement of the Jensen's inequality for this class of functions. A typical example is $f(x) = \dfrac{1}{\sin x}$ in $(0,\pi/2)$. Clearly, $f'(x) = -\dfrac{\cos x}{\sin^2x} \le 0$ and 
	$$\frac{d^2}{dx^2}\log \left(f(e^x)\right)=\frac{e^x\left(e^x-\frac{1}{2}\sin 2x\right)}{\sin^2x} > 0$$ in $(0,\pi/2)$, since $e^x> 1$ and $\sin 2x \le 1$ in $(0,\pi/2)$.
	\end{itemize}
	\end{remark}

	Before closing this section, we state an interesting theorem that illustrates an inequality by monotonicity and convexity/concavity. We need the following well-known lemma as Jensen's operator inequality for a convex function.
	\begin{lemma}{\bf (\cite{mond1995jensen},\cite[Theorem 1.2]{pecaric2005mond})}\label{unproved_lemma}
		Let $f$ be a concave function on $[a, b]$ where $a \geq 0$ and $A$ be a self-adjoint operator with its spectrum in $[a, b]$. Then we have
		$$\langle f(A)x, x \rangle \leq f \left(\langle Ax, x \rangle\right)$$
		for every unit vector $x\in {\cal H}$. If $f$ is a convex function, then the revered inequality above holds.
	\end{lemma}
		\begin{theorem}\label{concave_and_convex_lemma}
			Let $a\ge 0$, $f:[a,b]\to [0,\infty)$ and $g:[a,b]\to [0,\infty)$ be differentiable functions satisfying the following (i)-(iv):
			\begin{enumerate}
				\item[(i)] 
					$f$ is monotone increasing and concave;
				\item[(ii)]
					$g$ is convex;
				\item[(iii)]
					$f(x) \geq g(x)$ for all $x \in [a, b]$; and 
				\item[(iv)]
					$f(b) - f(a) \geq {\rm M_{ratio}}(f,g) (g(b) - g(a)) \geq 0$, where ${\rm M_{ratio}}(f,g):={{\max\limits_{x \in [a, b]}\{f(x)\}} \mathord{\left/
 {\vphantom {{\max\limits_{x \in [a, b]}\{f(x)\}} {\min\limits_{x \in [a, b]}\{g(x)\}}}} \right.
 \kern-\nulldelimiterspace} {\min\limits_{x \in [a, b]}\{g(x)\}}}.$
			\end{enumerate}
			Then, for any positive operator $A$ with $a{{\mathbf 1}_{\mathcal H}} \leq A \leq b{{\mathbf 1}_{\mathcal H}}$, and a unit vector $h \in {\cal H}$ we have
			$$(g(b) - g(a)) f(\langle Ah , h \rangle ) \leq (f(b) - f(a)) \langle g(A)h, h \rangle.$$  
		\end{theorem}	
		
		\begin{proof}
			Define a function $F:[a, b] \rightarrow \mathbb{R}$ by
			\begin{equation}
				F(t) = (f(b) - f(a))g(t) - (g(b) - g(a)) f(t).
			\end{equation} 
			From the conditions (i),(ii) and (iv), we see that $F(t)$ is a convex function.	
			As $f(t)$ and $g(t)$ are differentiable functions, 
			\begin{equation}
				F'(t) = (f(b) - f(a))g'(t) - (g(b) - g(a)) f'(t).
			\end{equation}
			The Cauchy mean value theorem indicates that there exists $\xi \in (a, b)$ such that
			\begin{equation}
				\begin{split}
					& \frac{g(b) - g(a)}{f(b) - f(a)} = \frac{g'(\xi)}{f'(\xi)}\\
				\text{or}~ & (g(b) - g(a))f'(\xi) = (f(b) - f(a))g'(\xi) \\
					\text{or}~ & (f(b) - f(a))g'(\xi) - (g(b) - g(a))f'(\xi) = 0, ~\text{or}~ F'(\xi) = 0.
				\end{split}
			\end{equation} 
			As $F(t)$ is a convex function, $\xi$ is minima of $F(t)$. Also, we have 
			\begin{equation}
				f(b) - f(a) \geq {\rm M_{ratio}}(f,g) (g(b) - g(a)) \geq 0.
			\end{equation}
			Hence, 
			\begin{equation}
				f(b) - f(a) \geq (g(b) - g(a)) \frac{f(\xi)}{g(\xi)}.
			\end{equation}
			Thus we have $F(\xi) = (f(b) - f(a))g(\xi) - (g(b) - g(a)) f(\xi) > 0$. Therefore, $F(t) \geq 0$ for all $t \in [a, b]$, that is
			\begin{equation}\label{first_fundamental}
				(g(b) - g(a)) f(t) \leq (f(b) - f(a)) g(t).
			\end{equation}
			Consider $t = \langle Ah, h \rangle$. As there exists $a$ and $b$ such that $a{{\mathbf 1}_{\mathcal H}} \leq A \leq b{{\mathbf 1}_{\mathcal H}}$, we have $a \leq t \leq b$. Replacing it in equation (\ref{first_fundamental}) we have
			\begin{equation}
				(g(b) - g(a)) f(\langle Ah, h \rangle) \leq (f(b) - f(a)) g(\langle Ah, h \rangle).
			\end{equation}
			As $g$ is a convex function applying Lemma \ref{unproved_lemma} we have
			\begin{equation}
				(g(b) - g(a)) f(\langle Ah, h \rangle) \leq (f(b) - f(a)) g(\langle Ah, h \rangle) \leq (f(b) - f(a)) (\langle g(A)h, h \rangle),
			\end{equation}
			which concludes this theorem.
		\end{proof}

\begin{corollary}\label{Added_cor01}
Let $f$ and $g$ satisfies the assumptions in Theorem \ref{concave_and_convex_lemma}. Then, for any positive operator $A,B$ with $aA\le B\le bA$,  we have
\[A^{1/2}f\left(A^{-1/2}BA^{-1/2}\right)A^{1/2}\le \left(\frac{f\left( b \right)-f\left( a \right)}{g\left( b \right)-g\left( a \right)}\right)A^{1/2}g\left(A^{-1/2}BA^{-1/2}\right)A^{1/2}.\]
\end{corollary}		

\begin{proof}
From the convexity of $f$ and $a{{\mathbf 1}_{\mathcal H}}\le A^{-1/2}BA^{-1/2}\le b{{\mathbf 1}_{\mathcal H}}$, we have with  Lemma \ref{unproved_lemma} and Theorem \ref{concave_and_convex_lemma},
$$
\left\langle {f\left( {{A^{ - 1/2}}B{A^{ - 1/2}}} \right)h,h} \right\rangle  \le \left( {\frac{{f\left( b \right) - f\left( a \right)}}{{g\left( b \right) - g\left( a \right)}}} \right)\left\langle {g\left( {{A^{ - 1/2}}B{A^{ - 1/2}}} \right)h,h} \right\rangle,
$$
which implies the result.
\end{proof}
\begin{corollary}\label{Added_cor02}
Let $f$ and $g$ satisfies the assumptions in Theorem \ref{concave_and_convex_lemma}. Given two positive operators $A$ and $B$ with $B\le A$ we have
\[f\left( B \right)\le \frac{f\left( b \right)-f\left( a \right)}{g\left( b \right)-g\left( a \right)}g\left( A \right).\]
\end{corollary}	
\begin{proof}
As we have shown in Theorem \ref{concave_and_convex_lemma}
\[f\left( \left\langle Ah,h \right\rangle  \right)\le \frac{f\left( b \right)-f\left( a \right)}{g\left( b \right)-g\left( a \right)}\left\langle g\left( A \right)h,h \right\rangle \]
for any unit vector $h\in\mathcal H$. Since $B\le A$, then $\left\langle Bh,h \right\rangle \le \left\langle Ah,h \right\rangle $ for any unit vector $h$. On the other hand, $f$ is an increasing function, thus $f\left( \left\langle Bh,h \right\rangle  \right)\le f\left( \left\langle Ah,h \right\rangle  \right)$. Therefore,
\[f\left( \left\langle Bh,h \right\rangle  \right)\le \frac{f\left( b \right)-f\left( a \right)}{g\left( b \right)-g\left( a \right)}\left\langle g\left( A \right)h,h \right\rangle.\]
Now, by Lemma \ref{unproved_lemma},
\[\left\langle f\left( B \right)h,h \right\rangle \le \frac{f\left( b \right)-f\left( a \right)}{g\left( b \right)-g\left( a \right)}\left\langle g\left( A \right)h,h \right\rangle,\]
for any unit vector $h\in\mathcal H$. By replacing $h={x}/{\left\| x \right\|}\;$, we reach the desired inequality.
\end{proof}

	\section{Applications to relative operator entropies}

		We apply the results discussed in section \ref{sec2} in the theory of relative operator entropy. Recall that, the $t$-logarithmic function $\ln_t$ is defined by $\ln_t(x) :=\dfrac{x^t-1}{t}$ for $x>0$ and $t\ne 0$. The deformed logarithm and their generalizations are widely utilized in mathematical inequality and information theory \cite{8,dutta_furuichi, dutta_furuichi_guha}. The function $\ln_t(x)$ is an inverse of $\exp_t(x):=(1+tx)^{1/t}$, which is defined for $1+tx>0$. We easily find that $\lim\limits_{t\to 0}\exp_t(x)=\exp(x)$ and $\lim\limits_{t\to 0}\ln_t(x)=\log(x)$. Given positive operators $A$, and $B$ the relative operator entropy $S(A|B)$ \cite{FK}, the Tsallis relative operator entropy $T_t(A|B)$ \cite{YKF2005}, and the generalized relative operator entropy $S_t(A|B)$ \cite{Furuta} are respectively defined by 	
		\begin{eqnarray}
			&S(A|B):=A^{1/2}\log \left(A^{-1/2}BA^{-1/2}\right)A^{1/2},\\
			&T_t(A|B):=A^{1/2}\ln_t \left(A^{-1/2}BA^{-1/2}\right)A^{1/2},\\
			\text{and}~ & S_t(A|B):=A^{1/2}\left(A^{-1/2}BA^{-1/2}\right)^t\log\left(A^{-1/2}BA^{-1/2}\right)A^{1/2}.
		\end{eqnarray}
		It is easy to observe that $\lim\limits_{t\to 0}T_t(A|B)=S(A|B)$ and  $\lim\limits_{t\to 0}S_t(A|B)=S(A|B)$ by $\lim\limits_{t\to 0}\ln_t x =\lim\limits_{t\to 0}\ln_{-t} x=\log x$.  In addition, the following inequalities are known \cite{zou2015operator},
		\begin{equation}\label{Zou_ineq}
			A-AB^{-1}A\le T_{-t}(A|B)\le S(A|B)\le T_t(A|B)\le B-A,
		\end{equation}
		where $t\in(0,1]$. 
		
		Note that, inequality \eqref{ineq01_proof_theorem2.1} is equivalent to
		$$\ln_{-t}x\le \log x\le \ln_tx$$%,\,\,\,(0<t\le 1)$$
		when $t:=\dfrac{1}{n}$.  The above inequality implies
		\begin{equation}\label{ordering_S_T_t}
			T_{-t}(A|B)\le S(A|B) \le T_t(A|B)%,\,\,\,(0<t\le 1)
		\end{equation}
		which are covered by \eqref{Zou_ineq}. See \cite[Section 7.3]{8} for these details.

	\begin{lemma}\label{10}
		The $t$-logarithmic function $\ln_tx$, is convex in $t$ for $x \geq 1$, and concave in $t$ for $0<x \le 1$. Furthermore, this function is log-convex on $\left( -\infty,\infty \right)$. Moreover, $\ln_tx$ is a monotone increasing function on $t\in \left( -\infty,\infty \right)$.
	\end{lemma}
	\begin{proof}
		We define the function $f(t):=t(\log t)^2-2t\log t+2t-2$ for $t>0$.
		Since $f(1)=0$ and $f'(t)=(\log t)^2 \geq 0$, we have
		$f(t)\le 0$ for $0<t\le 1$ and $f(t)\ge 0$ for $t \ge 1$.
		
		Note that
		$$
		\frac{d^2}{dt^2}\left(\ln_t x\right)=\frac{1}{t^3}\left\{x^t(\log x^t)^2-2x^t\log x^t+2x^t-2\right\} = \frac{f(x^t)}{t^3}.
		$$
		Hence, $\dfrac{d^2}{dt^2}\left(\ln_t x\right)\ge 0$ if $x \geq 1$ and $\dfrac{d^2}{dt^2}\left(\ln_t x\right)\le  0$ if $0<x \le 1$.
		
		We also calculate as
		$$
		\frac{d^2}{dt^2}\left(\log \left(\ln_t x\right)\right) =\frac{\left(x^t-1\right)^2-x^t\left(\log x^t \right)^2}{t^2\left(x^t-1\right)^2}\geq 0.
		$$
		The above inequality holds by the relation $\sqrt{a}\le \dfrac{a-1}{\log a}$ for $a>0$. Thus the function $\ln_tx$ is log-convex for all $t\in \mathbb{R}$.
		
		Finally, we can calculate
		$$
		\frac{d}{dt}\left(\ln_tx\right)=\frac{x^t\log x^t -(x^t-1)}{t^2}\geq 0.
		$$
		The above inequality can be easily proven by using the fundamental inequality $\log a \le a-1$ for $a>0$.
	\end{proof}

	Employing Lemma \ref{10} and Remark \ref{3} (i), we have:
	\begin{lemma}\label{lemma_5}
		Given $x> 0$, and $0\le t\le 1$ 
		\[\log x+\frac{{{\theta }(t,x)}}{t}\le {{\ln }_{t}}x,\]
		where ${{\theta }}\left( t,x \right)=\min \left\{ {{\left( \ln_tx-\log x \right)}^{2}},{{t}^{2}} \right\}\ge 0$.
	\end{lemma}
	
	Note that $\lim\limits_{t\to 0}\dfrac{\left( \ln_tx-\log x \right)^{2}}{t}=0$. Therefore, $\lim\limits_{t\to 0}\dfrac{\theta(t,x)}{t}=0$. Thus $\dfrac{\theta(t,x)}{t}$ is defined in $t=0$.
	 We have the following result from Lemma \ref{lemma_5}
	 \begin{theorem}\label{refined_ordering_S_T_t}
	 Let $A$, and $B$ be positive operators, such that, $mA\le B \le MA$ with $0<m\le M$, and let $0\le t \le 1$. Then we have
	 \begin{itemize}
	 \item[(i)] If $M<1$, then $S(A|B)+\dfrac{\theta(t,M)}{t}A\le T_t(A|B)$.
	 \item[(ii)] If $m\le 1 \le M$, then $S(A|B)\le T_t(A|B)$.
	 \item[(iii)]If $1<m$, then $S(A|B)+\dfrac{\theta(t,m)}{t}A\le T_t(A|B)$.
	 \end{itemize}
	 \end{theorem}
	 
	 \begin{proof}
	 Note that  the function $f_t(x):=\ln_tx-\log x$ defined for $x > 0$ and $0\le t \le 1$ is monotone decreasing in $0< x \le 1$ and monotone increasing in $x \ge 1$. From Lemma \ref{lemma_5}, for the case (i) we have
	 $$
	 \log x +\min_{(0<)m\le x \le M (<1)}\frac{\theta(t,x)}{t}\le \ln_t x.
	 $$
Thus we have the inequality in (i)  by setting $x:=A^{-1/2}BA^{-1/2}$ and multiplying $A^{1/2}$ to both sides.
The inequalities in (ii) and (iii) similarly follow with $\theta(t,1)=0$. 
	 
	 \end{proof}
	Theorem \ref{refined_ordering_S_T_t} includes a refinement of the second inequality in \eqref{ordering_S_T_t}.
	
	\begin{lemma}\label{lemma_7}
		For $s,t>0$ and $x\ge 1$, we have
		$$
		\exp \left(\frac{\left(x^s\log x-\ln_sx\right)\left( t-s \right)}{s\ln_sx}  \right)\ln_sx\le \ln_tx \le \exp \left(\frac{\left(x^t\log x-\ln_tx\right)\left( t-s \right)}{t\ln_tx}  \right)\ln_sx.
		$$
	\end{lemma}
	\begin{proof}
	Observe that $\ln_t x \ge 0$ if $x\ge 1$ and $t>0$. By Lemma \ref{10}, $\ln_t x$ is log-convex in $t> 0$. Then we can apply the inequalities \eqref{theorem3_1st_statement} in Theorem \ref{1} with $f(t):=\ln_t x$.
	\end{proof}
	The following result gives the relation of two Tsallis relative operator entropies.
	\begin{theorem}
	Let $A,B$ be positive operators such that $mA\le B \le MA$ with $1\le m\le M$, and let $s,t> 0$.
	Then we have
	\begin{itemize}
	\item[(i)] If $t\ge s \ge 0$, then $0\le  e^{\eta(m,s)(t-s)}T_s(A|B)\le T_t(A|B)\le e^{\eta(M,t)(t-s)}T_s(A|B)$.
	\item[(ii)]  If $s\ge t \ge 0$, then $0\le   e^{\eta(M,s)(t-s)}T_s(A|B)\le T_t(A|B)\le e^{\eta(m,t)(t-s)}T_s(A|B)$.
	\end{itemize}
	where $\eta(x,a):=\dfrac{x^a\log x-\ln_ax}{a\ln_ax}$ is defined  for $a\ge 0$ and $x>0$.
	\end{theorem}
	\begin{proof}
	Consider the function $\eta(x,a)=\dfrac{x^a\log x^a-x^a+1}{a(x^a-1)}$ for $a\ge 0$ and $x>0$. Then we have $\dfrac{d\eta(x,a)}{dx}=\dfrac{x^{a-1}\left(x^a-1-\log x^a\right)}{(x^a-1)^2}\ge 0$ with $\lim\limits_{x\to 1}\eta(x,a)=0$. Therefore the function $\eta(x,a)$ is monotone increasing in $x>0$ for all $a \ge 0$, and $\eta(x,a)\le 0$ if $0<x\le 1$, and $\eta(x,a)\ge 0$ if $x \ge 1$. For the case $t\ge s \ge 0$ and $m \ge 1$, we have $0\le \exp\left(\eta(m,s)(t-s)\right)\ln_s x\le\ln_t x \le \exp\left(\eta(M,s)(t-s)\right)\ln_s x$ thanks to Lemma \ref{lemma_7}. Putting $x:=A^{-1/2}BA^{-1/2}$ in these inequalities and multiplying $A^{1/2}$ to both sides, then we have the operator inequalities given in (i). Taking an attention that for $a \ge 0$ we have $\ln_ax \ge 0$ if $x \ge 1$,
	(ii) can be proven similarly using  Lemma \ref{lemma_7}. 
	\end{proof}

	We give another application of Theorem \ref{1} for estimating bounds of the relative operator entropy.
	\begin{theorem}\label{bound_ROE}
	Let $A,B$ be positive operators such that $mA\le B \le MA$ with $0<m\le M$.
	\begin{itemize}
	\item[(i)] If $M\le 1/e$, then we have
	\begin{equation}\label{bound_ROE_eq01}
	-\exp\left(\dfrac{em-1}{em\log m}\right)A\le S(A|B)\le -\exp\left(1-eM\right)A\le 0.
	\end{equation}
	\item[(ii)] If $1\le m\le M\le e$, then we have
	\begin{equation}\label{bound_ROE_eq02}
	0\le \exp\left(\dfrac{e-M}{M\log M}\right)A \le S(A|B)\le \exp\left(\dfrac{e-m}{e}\right)A.
	\end{equation}
	\end{itemize}
	\end{theorem}
	\begin{proof}
	\begin{itemize}
	\item[(i)] Since $\dfrac{d^2}{dx^2}\left(\log \left(-\log x\right)\right)=\dfrac{-1-\log x}{x^2\left(\log x\right)^2} \ge 0$ for $x\in (0,1/e]$, the function $f(x):=-\log x$ is log-convex. Also we observe that $f(x)\ge 0$ for $x\in (0,1/e]$. Applying the inequalities \eqref{theorem3_1st_statement} in Theorem \ref{1} with $f(x):=-\log x$, $J:=(0,1/e]$ and $s:=1/e$ and manipulating its inequalities, we have
	$$
	\exp(1-et)\le -\log t\le \exp\left(\dfrac{t-1/e}{t\log t}\right).
	$$
 	Putting $t:=A^{-1/2}BA^{-1/2}$ in the above, from the condition $mA\le B \le MA$ with $0<m\le M$, we have
 	$$
 	\min_{(0<)m\le t \le M (\le 1/e)}\exp(1-et)\le -\log A^{-1/2}BA^{-1/2} \le \max_{(0<)m\le t \le M (\le 1/e)}\exp\left(\dfrac{et-1}{et\log t}\right).
 	$$
 	Multiplying $A^{1/2}$ to both sides, we have the inequalities \eqref{bound_ROE_eq01}, since
 	the function $1-et$ is decrasing and the function $\dfrac{et-1}{et\log t}$ is decreasing on $(0,1/e]$.
 	Indeed, we calculate $\dfrac{d}{dt}\left(\dfrac{et-1}{et\log t}\right)=\dfrac{h(t)}{e t^2(\log t)^2}$ with $h(t):=\log t -e t +1$. Since $h'(t)=1/t-e\ge 0$ so that $h(t)\le h(1/e)=-1$, we have $\dfrac{d}{dt}\left(\dfrac{et-1}{et\log t}\right) < 0$.
	\item[(ii)] Since $\dfrac{d^2}{dx^2}\left(-\log \left(\log x\right)\right)=\dfrac{\log x+1}{x^2(\log x)^2}\ge 0$ for $x \ge 1$. Also it is trivial that $\log x \ge 0$ for $x \ge 1$. Applying the inequalities \eqref{theorem3_2nd_statement} in Theorem \ref{1} with $f(x):=\log x$, $J:=[1,\infty)$ and $s:=e$ and manipulating its inequalities, we have
	$$
	\exp\left(\dfrac{e-t}{t\log t}\right)\le \log t\le \exp\left(\dfrac{e-t}{e}\right).
	$$
	In a similar way to (i), we have
	$$
	\min_{(1\le)m\le t \le M}\exp\left(\dfrac{e-t}{t\log t}\right)\le \log A^{-1/2}BA^{-1/2} \le \max_{(1\le)m\le t \le M}\exp\left(\dfrac{e-t}{e}\right),
	$$
	which implies  the inequalities \eqref{bound_ROE_eq01}, since the function $\dfrac{e-t}{e}$ is decreasing and the function $\dfrac{e-t}{t\log t}$ is also decreasing for $t \in [1,e]$. Indeed, we calculate $\dfrac{d}{dt}\left(\dfrac{e-t}{t\log t}\right)=\dfrac{k(t)}{t^2(\log t)^2}$ with $k(t):=t-e-e\log t$. Since $k'(t)=-e/t+1\le 0$ for $1\le t \le e$ so that $k(t)\le k(1)=1-e<0$, we have $\dfrac{d}{dt}\left(\dfrac{e-t}{t\log t}\right)<0$. 
	\end{itemize}
	\end{proof}

	\begin{lemma}\label{5}
		Let $f:\mathbb{R}\to(0,\infty)$ be a log-convex function. Then the functions $g(t):=\dfrac{f(t)}{(1-t)f(0)+t f(1)}$ and $h(t):=\dfrac{f(t)}{f(0)^{1-t}f(1)^t}$ for $t\in\mathbb{R}$ are log-convex too.
	\end{lemma}
	\begin{proof}
		Since $f$ is log-convex, we have $\dfrac{d^2}{dt^2}\log f(t)=\dfrac{f''(t)f(t)-f'(t)^2}{f(t)^2}\geq 0$, namely $f''(t)f(t)-f'(t)^2 \ge 0$, then we obtain
		\begin{eqnarray*}
			&&\frac{{{d^2}}}{{d{t^2}}}\left( {\log \left( {\frac{{f\left( t \right)}}{{\left( {1 - t} \right)f\left( 0 \right) + tf\left( 1 \right)}}} \right)} \right)\\
			&& = \frac{{f{{\left( t \right)}^2}{{\left\{ {f\left( 1 \right) - f\left( 0 \right)} \right\}}^2} + \left\{ {f''\left( t \right)f\left( t \right) - f'{{\left( t \right)}^2}} \right\}{{\left\{ {\left( {1 - t} \right)f\left( 0 \right) + tf\left( 1 \right)} \right\}}^2}}}{{f{{\left( t \right)}^2}{{\left\{ {\left( {1 - t} \right)f\left( 0 \right) + tf\left( 1 \right)} \right\}}^2}}} \ge 0,
		\end{eqnarray*}
		which shows log-convexity of  ${g(t)}$.
		
		We have
		\begin{equation*}
		h\left(\frac{t+s}{2}\right) = \frac{f\left(\frac{t+s}{2}\right)}{f(0)^{1-\frac{t+s}{2}}f(1)^{\frac{t+s}{2}}} \leq\frac{\sqrt{f(t)f(s)}}{\sqrt{f(0)^{1-t}f(1)^t}\sqrt{f(0)^{1-s}f(1)^s}} =\sqrt{h(t)h(s)},
		\end{equation*}
		which shows log-convexity of  ${h(t)}$.
	\end{proof}

	\begin{corollary}
		Let $a,b\ge 0$ and $0\le t\le 1$. Then
		\begin{equation}\label{cor_5_eq01}
		\exp \left( \left( \log \frac{b}{a}+1-\frac{b}{a} \right)t \right) \le \frac{a{{\sharp }_{t}}b}{a{{\nabla }_{t}}b} \le \exp \left( \left(\log \frac{b}{a}-\dfrac{b-a}{a\nabla_t b}\right)t \right).
		\end{equation}
	\end{corollary}
	\begin{proof}
		In Lemma \ref{5}, we take $f(t):=a\sharp_t b$. Then the function $g\left( t \right)=\dfrac{f(t)}{(1-t)f(0)+tf(1)}=\dfrac{a{{\sharp }_{t}}b}{a{{\nabla }_{t}}b}$ is log-convex. Since we have
		$g\left( 0 \right)=1$, 
		$g'\left( t \right)=g(t)\left(\log\dfrac{b}{a} -\dfrac{(b-a)}{a\nabla_t b}\right)$ and
		$g'\left( 0 \right)=\log \dfrac{b}{a}+1-\dfrac{b}{a}$,
		we get the desired result by Theorem \ref{1} with $s=0$.
	\end{proof}
	
	\begin{remark}
	\begin{itemize}
\item[(i)]	Putting $a=1$, $b=x$ and taking logarithm of both side of the first ineauality in \eqref{cor_5_eq01}, we have $\log \left\{(1-t)+tx\right\}\le t (x-1)$. Putting $x:=A^{-1/2}BA^{-1/2}$ in this inequality and  multiplying $A^{1/2}$ to both sides, then we have
	$$
	A^{1/2}\log\left((1-t)I+tA^{-1/2}BA^{-1/2}\right)A^{1/2} \le t(B-A),
	$$
	which includes the known inequality (see e.g.,\cite[Eq.(7.3.1)]{8}), $S(A|B)\le B-A$ when $t=1$.
	\item[(ii)] The inequalities \eqref{cor_5_eq01} is rewitten as
			\begin{equation}\label{cor_5_eq02}
		\exp \left( \left(\dfrac{b-a}{a\nabla_t b}  -\log \frac{b}{a}\right)t \right) 
		\le \frac{a{{\nabla }_{t}}b}{a{{\sharp }_{t}}b} 
		\le \exp \left( \left( \frac{b}{a}-1-\log \frac{b}{a}    \right)t \right).
		\end{equation}
		Since we have $\log x \le x-1$ for all $x>0$, $\exp \left( \left( \frac{b}{a}-1-\log \frac{b}{a}    \right)t \right)\ge 1$ for $a,b \ge 0$ and $0\le t \le 1$, and the second inequality of \eqref{cor_5_eq02} gives the reverse of the Young's inequality. On the other hand, the function $\varphi(t,x):=\dfrac{x-1}{(1-t)+tx}-\log x$ takes both positive and negative values. For example $\varphi(1/4,4)\simeq 0.327991$ and $\varphi(1/2,4)\simeq -0.186294$. Therefore the first inequality of  \eqref{cor_5_eq02} does not always give the refinement of the Young's inequality. However, we say that the first inequality of  \eqref{cor_5_eq02}  gives the refinement of the Young's inequality for $a \ge b >0$ and $0\le t \le 1/2$. Indeed, the function $\varphi(t,x)$ is monotone decreasing in $t\in [0,1]$ by
		$\dfrac{d\varphi(t,x)}{dt}=-\dfrac{(x-1)^2}{\left\{(1-t)+tx\right\}^2}\le 0$. If $0\le t \le 1/2$ and $0<x \le 1$, then we have $\varphi(t,x)\ge \varphi(1/2,x)=\dfrac{2(x-1)}{x+1}-\log x \ge 0$ since we have $\dfrac{x-1}{\log x}\le \dfrac{x+1}{2}$ for all $x>0$.
	\end{itemize}
	\end{remark}
			
\begin{proposition}{\bf (\cite[Theorem 1]{IIKTW2012})}
For positive operators $A,B$ and $p>0$, we have
$S(A|B)\le T_p(A|B)\le S_p(A|B)$
	 and their reverse inequalities hold when $p<0$.
\end{proposition}
\begin{proof}
	Putting $t=1$ in \eqref{theorem_9_eq01}, we have
	$$
	b\left(\log\frac{a}{b}\right)\frac{g'(b)}{g(b)}\le \log\frac{g(a)}{g(b)}\le a\left(\log\frac{a}{b}\right)\frac{g'(a)}{g(a)}.
	$$
	 Taking a geometrically convex function $g(x):=\exp\left(x^p\right),\,\,\,(p\neq 0)$ in the above as an example, we have
	 $$
	 b^p\left(\log\frac{a^p}{b^p}\right)\le a^p-b^p\le a^p\left(\log\frac{a^p}{b^p}\right),\,\,\,(a,b>0,\,\,p\neq 0),
	 $$
	 which implies
	 $$
	 b^p\left(\log\frac{a}{b}\right)\le \ln_p a-\ln_p b\le a^p\left(\log\frac{a}{b}\right),\,\,\,(a,b>0,\,\,p>0)
	 $$
	 and
	 $$
	 	 b^p\left(\log\frac{a}{b}\right)\ge \ln_p a-\ln_p b\ge a^p\left(\log\frac{a}{b}\right),\,\,\,(a,b>0,\,\,p<0).
	 $$
	 The above inequalities give the relations among relative operator entropies for $A, B\ge 0$:
	 $$
	 S(A|B)\le T_p(A|B)\le S_p(A|B),\,\,\,(p>0)
	 $$
	 and
	 $$
	 S(A|B)\ge T_p(A|B)\ge S_p(A|B),\,\,\,(p<0)
	 $$
	 by  taking $b=1$, $a=x=:A^{-1/2}BA^{-1/2}$ and multiplying $A^{1/2}$ to both sides.
		\end{proof}

		 We have the following results on the bounds of the	Tsallis relative operator entropy 	$T_t(A|B)$. The following application of Corollary \ref{Added_cor01} is a simple case such as $f(a)=g(a)$ and $f(b)=g(b)$, namely	$\dfrac{f\left( b \right)-f\left( a \right)}{g\left( b \right)-g\left( a \right)}=1$.
		 \begin{theorem}\label{TROE_result01}
		 For two positive operators $A,B$ such that $mA\le B \le MA$, where $1\le m < M$,
		 we have
		 \begin{equation}\label{TROE_result01_ineq01}
		 T_t(A|B)\le \left(\frac{\ln_tM-\ln_tm}{M-m}\right)B+\left(\frac{M\ln_tm-m\ln_tM}{M-m}\right)A,\quad (t\le 1,\,\,t\neq 0).
		 \end{equation}
		 We also have the reverse inequality of \eqref{TROE_result01_ineq01} for $t \ge 1$.
		 \end{theorem}
		 
		 \begin{proof}
		 If we take $f(x)=\ln_t x$, $(t\le 1,\,\,t\neq 0)$ and $g(x)=\left(\frac{\ln_tM-\ln_tm}{M-m}\right)(x-m)+\ln_tm$ on $[m,M]$ where $m\ge 1$, then these functions satisfy the conditions (i)--(iv) in Theorem \ref{concave_and_convex_lemma}. Since $1\le m \le A^{-1/2}BA^{-1/2}\le M$, we have the inequality \eqref{TROE_result01_ineq01}, by Corollary \ref{Added_cor01}.
		 		 Since the functions $f(x)=\left(\frac{\ln_tM-\ln_tm}{M-m}\right)(x-m)+\ln_tm$ and $g(x)=\ln_t x$, $(t\ge 1)$ on $[m,M]$ where $m \ge 1$ also satisfy the conditions (i)--(iv) in Theorem \ref{concave_and_convex_lemma}, we then have the reverse inequality of \eqref{TROE_result01_ineq01} for $t \ge 1$, by Corollary \ref{Added_cor01}.
		 \end{proof}
		 \begin{remark}
		 \begin{itemize}
		 \item[(i)] In \eqref{TROE_result01_ineq01}, we take $m=1$, then we have
		 \begin{equation}\label{TROE_result01_ineq02}
		 T_t(A|B)\le \frac{\ln_tM}{M-1}(B-A)\le B-A,\,\,\,\,(t \le 1,\,\, t\neq 0,\,\,M\geq 1).
		 \end{equation}
		 The last inequality holds since we have $\ln_t x \le x-1$,  $(x>0,\,\,t\le 1,\,\,t\neq 0)$.
		 Therefore the inequality \eqref{TROE_result01_ineq02} give an refinement for the upper bound of  the 	Tsallis relative operator entropy 	$T_t(A|B)$ given in \eqref{Zou_ineq}.
		 We also have
		 $$
		 T_t(A|B)\ge  \frac{\ln_tM}{M-1}(B-A) \ge  B-A,\,\,\,\,(t \ge 1, \,\,M\geq 1),
		 $$
		 since we have the inequality $\ln_tx \ge x-1$, $(x>0,\,\,t\ge 1)$.
		 Further, we put $M=e$ and take a limit $t\to 0$ in \eqref{TROE_result01_ineq02}, we have
		 the slightly refined upper bound for  the relative operator entropy 	$S(A|B)$: 
		 $$
		 S(A|B)\le\frac{1}{e-1}(B-A)\le B-A.
		 $$
		 \item[(ii)]	Since the function $f$ is concave in Theorem \ref{concave_and_convex_lemma}, we have
					$$(g(b) - g(a)) (\langle f(A)h , h \rangle ) \leq (f(b) - f(a)) \langle g(A)h, h \rangle$$
					by  Lemma \ref{unproved_lemma}.
					We take the functions $f(x):=\log x$ and $g(x):=\dfrac{x-1}{e-1}$ in the interval $[1,e]$ to satisfy the conditions in  Theorem \ref{concave_and_convex_lemma}. Then we have
					$$
					\langle \left(\log A\right) h,h\rangle \le \langle \left(\frac{A-I}{e-1}\right)h,h\rangle
					$$
					for any unit vector $h\in {\cal H}$. Thus we have the operator inequality $\log A \le \dfrac{A-{{\mathbf 1}_{\mathcal H}}}{e-1}$ for ${{\mathbf 1}_{\mathcal H}}\le A \le e{{\mathbf 1}_{\mathcal H}}$.
		 \end{itemize}
		 \end{remark}
	
\section*{Acknowledgement}
%\section*{Acknowledgements}
%The authors would like to thank the referees for their careful and insightful comments to improve our manuscript.
The author (S.F.) was partially supported by JSPS KAKENHI Grant Number 21K03341.

\end{document}